\documentclass[9pt]{article}
\usepackage{graphicx}
\usepackage{float}
\usepackage{amsfonts,amsmath,amssymb,amsbsy,amsthm,enumerate}
\usepackage{bm}

\usepackage{tikz}
\usetikzlibrary{positioning,chains,fit,shapes,calc}
\usetikzlibrary{calc,positioning,decorations.pathmorphing,decorations.pathreplacing}

\usetikzlibrary{decorations.markings}
\usetikzlibrary{decorations.pathmorphing}
\usetikzlibrary{arrows.meta}
\usetikzlibrary{positioning,chains,fit,shapes,calc}

\newcommand{\etal}{\textit{et~al.}}
\usepackage{mathrsfs} 
\usepackage{comment}
\usepackage[caption=false]{subfig}
\newtheorem{theorem}{Theorem}[section]

\newtheorem{conjecture}[theorem]{Conjecture}
\newtheorem{proposition}[theorem]{Proposition}
\newtheorem{lemma}[theorem]{Lemma}
\newtheorem{corollary}[theorem]{Corollary}

\newtheorem{claim}{Claim}

\definecolor{myblue}{RGB}{80,80,160}
\definecolor{mygreen}{RGB}{80,160,80}
\definecolor{myred}{RGB}{255,0,0}
\definecolor{mybrown}{RGB}{139,69,19}

\newtheoremstyle{case}{}{}{}{}{\bfseries}{:}{ }{}
\theoremstyle{case}

\numberwithin{subcase}{case}

\usepackage{marginnote}

\title{On the intersection of two longest paths in $k$-connected graphs}
\author{Juan Guti\'errez %\textsuperscript{1}\\
}

\begin{document}

\maketitle

\begin{abstract}
%Hippchen conjectured that every pair of longest paths in a $k$-connected graph intersect each other in at least $k$ vertices.
%In this paper, we show Hippchen conjecture for $k=4$ and show that 
%every pair of longest paths in a $k$-connected graph intersect each other in at least $\max\{\frac{6k-n}{3k},\frac{4k-n}{k}\}$
%$(6k-n+2)/3$ vertices
We show that every pair of longest paths in a $k$-connected graph on $n$ vertices intersect each other in at least 
%$\max\{\frac{6k-n}{3k},\frac{4k-n}{k}\}$
%$(6k-n+2)/3$ vertices.
$(8k-n+2)/5$ vertices.
We also show that, in a 4-connected graph, every pair of longest paths intersect each other in at least four vertices.
This confirms a conjecture of Hippchen for $k$-connected graphs when $k\leq 4$ or $k\geq (n-2)/3$.
%We show that in 4-connected graphs, every pair of longest paths intersect in at least four vertices. We also show, for every $k$, a family of $k$-connected graphs in which there is a pair of longest paths intersecting each other in exactly $k$ vertices.
\end{abstract}

\section{Introduction}

It is known that every pair of longest paths
in a connected graph intersect in at least a vertex.
Hippchen \cite[Conjecture 2.2.4]{Hippchen08} conjectured that, for $k$-connected
graphs, every pair of longest paths intersect each other in at least $k$ vertices.
A similar conjecture, for cycles instead of paths, was proposed by Gr\"otschel  and attributed to Scoot Smith \cite[Conjecture 5.2]{ Grotschel84}.

Smith's Conjecture has been verified up to $k=6$
\cite{Grotschel84}, and for a general $k$, it was proved that every pair of longest cycles intersect in at least $ck^{3/5}$ vertices, for a constant $c \thickapprox 0.2615$
\cite{Chen98}.
However, for Hippchen's Conjecture, the only known result is for $k=3$ and was proved by Hippchen himself
\cite[Lemma 2.2.3]{Hippchen08}.
In this paper, we verify Hippchen's Conjecture for $k=4$.

For $k\geq 5$, Hippchen's Conjecture seems not trivial and hard to prove.
As it is difficult to prove it, it is natural to ask for lower bounds on the intersection of two longest paths in $k$-connected graphs.
Note that, if the graph is Hamiltonian, then it is clear that this lower bound equals $k$. As Hamiltonian paths appear in highly connected graphs, this motivates us to study cases in which $k = \Omega(n)$.
%$k$ such lower bound, making Hippchen's conjecture true.
%A simple observation (Lemma ) shows us that
%if 
In this paper, we show that
any two longest paths intersect in at least $(8k-n+2)/5$ vertices.
In fact, this result says that if $k=\Omega(n)$, then every pair of longest paths intersect in at least $\Omega(k)$ vertices.

Finally, we show, for any $k$, an infinite family of graphs that make Hippchen's Conjecture tight.

%A similar conjecture, for cycles instead of paths, was proposed by Gr\"otschel  and attributed to Scoot Smith \cite[Conjecture 5.2]{ Grotschel84}.
%This conjecture has been verified up to $k=6$
%\cite{Grotschel84}, and for a general $k$, it was proved that every pair of longest cycles intersect in at least $ck^{3/5}$ vertices, for a constant $c \thickapprox 0.2615$
%\cite{Chen98}.

%\cite{Shabbir13}.
%Several results have been made on this conjecture
%\cite{,Chen98}.

%was attributed by Groetschel \cite{Groets. Much more results have been proven in this case 
%vertices and prove it for $k=3$.
%In this paper, we prove Hippchen's conjecture for $k=4$.

\section{Preliminaries}\label{section:definitions}

In this paper all graphs are simple (without loops or parallel edges) and the notation and terminology are standard.
When we refer to paths, we mean simple paths (without repetitions of edges or vertices).
The length of a path~$P$ is the number of edges it has, and it is denoted by~$|P|$. A longest path in a graph is a path with maximum length among all paths.
Given a path~$P$ and two vertices~$x$ and $y$ 
in~$P$, we denote by~$P[x,y]$ the subpath of~$P$ with extremes~$x$ and~$y$.
Also, we denote the length of~$P[x,y]$
by~$\mbox{dist}_P(x,y)$.

Given two set of vertices~$S$ and~$T$ in a graph~$G$, an~$S$-$T$ path is a path with one extreme in~$S$, the other extreme in~$T$,
and whose internal vertices are neither in~$S$ nor~$T$.
If~$S=\{v\}$, we also say that a~$S$-$T$ path is a~$v$-$T$ path.
When we refer to the intersection of two paths in a graph, we mean vertex-intersection, that is, the set of vertices the paths share.
Two paths are internally disjoint if
they have no internal vertices in common.
%their intersection does not contain any vertices different from their extremes.

A graph~$G$ is~$k$-connected if, for any two
vertices~$u$ and~$v$ in~$G$, there exists a set of
$k$~$u$-$v$ internally disjoint paths. 
Our proofs rely in two well-known facts, that we state in the following propositions. The first proposition is also known as Fan Lemma.
The second one is an easy corollary of
a similar result on cycles instead of paths
{\cite[Exercise 3.14]{Diestel10}}.

\begin{proposition}[{\cite[Proposition 9.5]{BondyM08}}]\label{prop:fan-lemma}
 Let~$G$ be a~$k$-connected graph. Let~${v \in V(G)}$
and~$S \subseteq V(G) \setminus \{v\}$.
If~$|S|\geq k$ then there exists a set of~$k$~$v$-$S$ internally disjoint paths.
Moreover, every two paths in this set have~$\{v\}$
as their intersection. 
\end{proposition}

%\begin{proposition}[Fan Lemma]\label{prop:fan-lemma}
%Let~$G$ be a~$k$-connected graph. Let~$v \in V(G)$
%and~$S \subseteq V(G) \setminus \{v\}$.
%If~$|S|\geq k$, then there exists a set of~$k$~$v$-$S$ internally disjoint paths.
%Moreover, every two paths in this set have~$\{v\}$
%as their intersection.
%\end{proposition}

%The following proposition is well-known (see Diestel, excercise).
%, for the sake of completion we include a proof of it.
\begin{proposition}\label{prop:diestel-exercise-paths}
The length of a longest path in a~$k$-connected no Hamiltonian graph is at least $2k$.
%If~$P$ is a longest path in a~$k$-connected no Hamiltonian graph, then~${|P| \geq 2k}$.
\end{proposition}

%\begin{proof}
%Let~$P$ be a longest path in~$G$.
%As~$G$ is not hamiltonian, there exists a vertex~$v \in V(G) \setminus V(P)$.
%By Proposition \ref{prop:fan-lemma}, there exists a set~$Q_1,Q_2,\ldots,Q_k$ of~$v$-$V(P)$ internally disjoint paths.
%Let~$w_1, w_2, \ldots w_n$ be the corresponding ends of the paths in~$P$, and suppose without loss of generality that they appear in this order in~$P$ (that is,~$~$Ṕ[w_iw_{i+1}]$
%\end{proof}

\section{High connectivity}
\label{section:basic-concepts-treewidth}

In this section, we show an interesting result for $k$-connected graphs when $k = \Omega(n)$.
We begin with a simple observation.

\begin{proposition}\label{prop:PcapQgeq2L+2-n}
Let $G$ be a $k$-connected graph on $n$ vertices.
Let $L$ be the length of a longest path in $G$. Let $P$ and $Q$ be two longest paths in $G$.
Then $|V(P)\cap V(Q)| \geq 2L+2-n$.
\end{proposition}
\begin{proof}
It suffices to note that $$|V(P)\cap V(Q)|=|V(P)|+|V(Q)|-|V(P)\cup	V(Q)|\geq 2L+2-n,$$ as we want.
\end{proof}

Proposition \ref{prop:PcapQgeq2L+2-n} together with Proposition \ref{prop:diestel-exercise-paths} are enough to give a nontrivial result on Hippchen's Conjecture.

\begin{corollary}
Let $G$ be a $k$-connected graph on $n$ vertices.
If $k \geq (n-2)/3$ then every two longest paths intersect in at least $k$ vertices.
\end{corollary}
%\begin{proof}
%Let $P$ and $Q$ be two longest paths in $G$.
%By Proposition \ref{prop:PcapQgeq2L+2-n}
%and \ref{prop:diestel},
%$|V(P)\cap V(Q)| \geq 2L+2-n \geq 4k+2-n\leq 3k+2$
%\end{proof}

Moreover, a strong result can be derived from these two propositions: every pair of longest paths intersect
in at least $4k+2-n$ vertices.
The rest of this section is devoted to improve this result when $k < \frac{n-2}{3}$.
The improvement relies in the following lemma. Its proof is presented at the end of the section.

\begin{lemma}\label{lemma:PcapQgeq3k-L}
Let $G$ be a $k$-connected graph.
Let $L$ be the length of a longest path in $G$. Let $P$ and $Q$ be two longest paths in $G$.
%Then $|V(P)\cap V(Q)| \geq 3k-L$.
Then $|V(P)\cap V(Q)| \geq 2k-L/2$.
\end{lemma}

With that lemma at hand, it is easy to obtain the main result of this section.

\begin{theorem}
Let $G$ be a $k$-connected graph on $n$ vertices. Let $P$ and $Q$ be two longest paths in $G$.
Then $|V(P)\cap V(Q)| \geq (8k-n+2)/5.$
\end{theorem}
\begin{proof}
Let $L$ be the length of a longest path in $G$.
By Proposition \ref{prop:PcapQgeq2L+2-n} and Lemma \ref{lemma:PcapQgeq3k-L},
we have
$$|V(P)\cap V(Q)|\geq \max\{2L+2-n, 2k-L/2\} \geq (8k-n+2)/5,$$ as we want.
%This result is attained when L=(4k+2n-4)/5
\end{proof}

%\begin{proof}
%Let $L$ be the length of a longest path in $G$.
%Note that $$|V(P)\cap V(Q)|=|V(P)|+|V(Q)|-|V(P)\cup	V(Q)|\geq 2L+2-n.$$
%Thus, by Lemma \ref{lemma:PcapQgeq3k-L},
%$$|V(P)\cap V(Q)|\geq \max\{2L+2-n, 3k-L\} \geq \frac{6k-n+2}{3}$$
%\end{proof}

%\begin{corollary}
%Let $G$ be a $k$-connected graph with $k\geq n/\alpha$. Let $P$ and $Q$ be two longest paths in $G$.
%Then $|V(P)\cap V(Q)| \geq (\frac{6-\alpha}{3})k$.
%\end{corollary}

We now proceed with the proof of
Lemma \ref{lemma:PcapQgeq3k-L}.
%\begin{lemma}\label{lemma:PcapQgeq3k-L}
%Let $G$ be a $k$-connected graph.
%Let $L$ be the length of a longest path in $G$. Let $P$ and $Q$ be two longest paths in $G$.
%Then $|V(P)\cap V(Q)| \geq 3k-L$.
%\end{lemma}
\begin{proof}[Proof of Lemma~\ref{lemma:PcapQgeq3k-L}]
Let~$X=V(P) \cap V(Q)$.
Let~$q$ be an extreme of~$Q$. 
Suppose for a moment that $q \in X$.
%As~$\delta(G) \geq k$,
As~$G$ is $k$-connected,~$q$ has at least~$k$ neighbors in~$Q$.
Let~$X'$ be the set of vertices of $Q$ that are next to a vertex in~$X$ considering the order of the path starting at~$q$.
That is,~$X'=\{x' \in V(Q): \text{ there exists a vertex } {x\in X} \text{ with } Q[q,x']=Q[q,x]+xx'\}$.
If every neighbor of~$q$
is in~$X'$, then, as~$L \geq 2k$
by Proposition \ref{prop:diestel-exercise-paths}, we have~$|X|\geq|X'| \geq k \geq 2k-L/2$ and we are done.

Hence, there exists a neighbor~$r$ of~$q$ in~$V(Q) \setminus X'$.
Let~$q'$ be the vertex adjacent to~$r$ in~$Q$ that is closer to~$q$ in~$Q$.
In that situation,
the path~$Q'=Q + rq - rq'$ is a longest path,
with~$q'$ as one of its extremes\footnote{This interchange is known as P\'{o}sa's rotation \cite{Posa76}.}.
As~$r \notin X'$, we have that~$q' \notin X$. %thus~$x \notin V(P)$.
Note that~$V(Q')=V(Q)$.
Thus, from now on, we may assume that
$q \notin X$.

By Proposition \ref{prop:fan-lemma}, as~$|V(P)| \geq k$, there exists
a set, say~$\mathcal{R}$, of~$k$~$q$-$V(P)$ internally disjoint paths that end at different vertices of $P$.
Let~$\mathcal{R}_{A}$ be the set of paths in~$\mathcal{R}$ that have an extreme in~$X$. That is,
$\mathcal{R}_{A}=\{R \in \mathcal{R}: V(R) \cap X \neq \emptyset\}$. Let~$\mathcal{R}_{B}=\mathcal{R} \setminus \mathcal{R}_{A}$.
Let~$A$ and~$B$ be the set of corresponding extremes of~$\mathcal{R}_{A}$ and~$\mathcal{R}_{B}$, respectively. That is,
$A= \{a \in X \cap V(R): R \in \mathcal{R}_A\}$
and~$B= \{b \in X \cap V(R): R \in \mathcal{R}_B\}$.

\begin{claim}\label{claim:RBgeq2}
 If~$R \in \mathcal{R}_{B}$, then~$|R| \geq 2$.
\end{claim}
\begin{proof}
Suppose by contradiction that~$|R|=1$. Let~$b$ be the extreme of~$R$ different from~$q$. As
$b \notin V(Q)$,~$Q + qb$ is a path in~$G$ longer than~$Q$, a contradiction.
\end{proof} 

%Note that, by (1),~$|A|=|X| \leq (\frac{6-\alpha}{3})k$.
%Also,~$k=|\mathcal{R}|=|A|+|B|$. Thus,
%\begin{equation}
%|B| \geq (\frac{\alpha-3}{3})k.
%\end{equation}

Let~$p_1$ and~$p_2$ be the two extremes of~$P$.
Let $\mathcal{R}=\{R_1,R_2, \ldots, R_k\}$, 
and for $1 \leq i \leq k$, let $v_i$
be the corresponding extreme of $R_i$ that is in $P$.
%Let $A \cup B = \{v_1,v_2, \ldots, v_k\}$
%and suppose, 
Moreover, we may assume, without loss of generality, that
$\text{dist}_P(p_1,v_i) < \text{dist}_P(p_1,v_{i+1})$ for $1 \leq i \leq k-1$.

%For each~$v \in A \cup B$, let~$v' \in A \cup B$ be the vertex in~$A \cup B \cup \{p_2\}$ that is next to~$v$ considering the ordering of~$P$ starting at~$p_1$. That is, for each~$v\in A \cup B$,~$v' \neq v$ is the vertex in $A \cup B \cup \{p_2\}$ with minimum value of $\mbox{dist}_P(p_1,v')$
%of~$|P[p_1,v']|$
%that satisfies
%%$P[p_1,v']=P[p_1,v]+P[v,v']$
%$\mbox{dist}_P(p_1,v')=\mbox{dist}_P(p_1,v)+%\mbox{dist}_P(v,v')$
%$d_P(v,v')=\min\{d(v,u): u \in P_{ve_2}-v\}$.
%(note that, as~$P$ is a longest path,~$p_2 \notin A \cup B$, so~$v'$ exists for every~$v$).

\begin{claim}\label{claim:p1R1p2Rk}
$\text{dist}_P(p_1,v_1) \geq |R_1|$ and 
$\text{dist}_P(v_k,p_2) \geq |R_k|$.
\end{claim}
\begin{proof}
It suffices to note that $P-P[p_1,v_1]+R_1$
and $P-P[v_k,p_2]+R_k$ are paths.
\end{proof}

\begin{claim}\label{claim:Pvivi+1-RiRi+1}
For $1 \leq i \leq k-1$, we have~$\mbox{dist}_P(v_i,v_{i+1}) \geq |R_i|+|R_{i+1}|$.
\end{claim}
\begin{proof}
It suffices to note that $P-P[v_i,v_{i+1}]+R_i+R_{i+1}$ is a path.
\end{proof}

Finally, by Claims \ref{claim:RBgeq2}, \ref{claim:p1R1p2Rk}
and \ref{claim:Pvivi+1-RiRi+1},
\begin{eqnarray*}
L &=& |E(P)| \\
&=& \text{dist}_P(p_1,v_1) + \sum_{i=1}^{k-1}\text{dist}_P(v_i,v_{i+1}) +  \text{dist}_P(v_k,p_2) \\
&\geq& |R_1| + \sum_{i=1}^{k-1}{(|R_i|+|R_{i+1}|)} + |R_k| \\
&=& 2\sum_{i=1}^{k}{|R_i|} \\
&=&
2\sum_{R \in \mathcal{R_A}}|R_A|+ 
2\sum_{R \in \mathcal{R_B}}|R_B| \\
&\geq& 2|A|+4|B| \\
%&=& 4|A|+4|B|-2|A| \\
&=& 4k-2|A| \\
&\geq& 4k-2|X|.
\end{eqnarray*}

Hence,
$$
|V(P)\cap V(Q)| =|X| \geq 2k-L/2,$$
as we want.
\end{proof}

%\begin{theorem}
%Let~$G$ be a $k$-connected graph on $n$ vertices. Let $P$ and $Q$ be two longest paths in $G$.
%Then $|V(P)\cap V(Q)| \geq (6k-n+2)/3.$
%\end{theorem}
%\begin{proof}
%Let $L$ be the length of a longest path in $G%$.
%Note that $$|V(P)\cap V(Q)|=|V(P)|+|V(Q)|-|V(P)\cup	V(Q)|\geq 2L+2-n.$$
%Thus, by Lemma \ref{lemma:PcapQgeq3k-L},
%$$|V(P)\cap V(Q)|\geq \max\{2L+2-n, 3k-L\} \geq \frac{6k-n+2}{3}$$

%\end{proof}

%\begin{corollary}
%Let $G$ be a $k$-connected graph with $k\geq n/\alpha$. Let $P$ and $Q$ be two longest paths in $G$.
%Then $|V(P)\cap V(Q)| \geq (\frac{6-\alpha}{3})k$.
%\end{corollary}

%\begin{corollary}
%If $G$ is a $k$-connected graph on $n$ vertices
%and $k\geq (n-2)/3$, then every pair of longest paths intersect each other in at least $k$ vertices.
%\end{corollary}

%\begin{corollary}
%If $G$ is a $xxx$-regular graph on $n$ vertices, then every pair of longest paths intersect each other in at least $k$ vertices.
%\end{corollary}

\section{Low connectivity}
\label{section:basic-concepts-treewidth}

In this section, we show Hippchen's Conjecture  for~$k=4$.
We begin by with a useful lemma.
\begin{lemma}\label{lemma:noDisjoint}
Let~$P$ and~$Q$ be two longest
paths in a graph~$G$.
Let~${u \in V(P) \cap V(Q)}$.
Let~$v$ be a vertex in 
${V(P) \setminus V(Q)}$
such that~$P[u,v]$
is internally disjoint from~$Q$.
Let~$w$ be a vertex in 
${V(Q) \setminus V(P)}$
such that~$Q[u,w]$
is internally disjoint from~$P$.
%Let~$w$ be a vertex in~$V(Q) \setminus \{u\}$
%such that~$Q[u,w]$ contains
%no vertex of~$V(P) \setminus \{u\}$.
Then, there is no~$vw$-path
internally disjoint from both~$P$
and~$Q$.
\end{lemma}
\begin{proof}
Suppose by contradiction that
there is a~$vw$-path~$R$
internally disjoint from~$P$
and~$Q$.
Note that~$P-P[u,v]+R+Q[u,w]$
and~$Q-Q[u,w]+R+P[u,v]$ are
both paths, whose lengths sum
$|P|+|Q|+2|R|$, a contradiction.
%(Figure \ref{figure:lemma-no-disjoint}).
\end{proof}

We now proceed to prove the main result of this section.
\begin{theorem}\label{theorem:4connected-4vertices}
Every pair of longest paths in a 4-connected graph intersect
in at least four vertices.
\end{theorem}
\begin{proof}
Let~$G$ be a 4-connected graph and let~$P$ and~$Q$ be two longest
paths in~$G$. Suppose by contradiction that~$P$ and~$Q$ do not
intersect in at least four vertices. As~$G$ is 3-connected,~$P$
and~$Q$ intersect in at least three vertices {\cite[Lemma 2.2.3]{Hippchen08}}. Hence,
$P$ and~$Q$ intersect in exactly three vertices, say~$a,b$ and~$c$. 
Let~$p_1$ and~$p_2$ be the extremes of~$P$.
Suppose, without loss of generality, that~$abc$ is a subsequence in~$P$ considering the sequence of~$P$ starting at~$p_1$.
That is,~$\text{dist}_P(p_1,a)<\text{dist}_P(p_1,b)<\text{dist}_P(p_1,c)$.

For simplicity of notation, we let~$P_a:=P[p_1,a]$,
$P_{ab}:=P[a,b]$,~$P_{bc}:=P[b,c]$ and
$P_{c}:=P[c,p_2]$.
We also let~$S:=\{a,b,c\}$,
$G':=G \setminus S$,
$P':=P \setminus S$,
$P'_a:=P_a \setminus S$,
$P'_{ab}:=P_{ab} \setminus S$,
$P'_{bc}:=P_{bc} \setminus S$ and
$P'_c:=P_c \setminus S$.
Without loss of generality, we have two cases,
depending on the order in which~$a,b$ and~$c$
appear in~$Q$.
In each of these cases, we assume similar notation
to the subpaths of~$Q$ and~$Q \setminus S$ as
we did for~$P$.

\begin{description}

\item [Case 1.]
$abc$ is a subsequence in~$Q$.

It is easy to see that
$|P_a|=|Q_a|$,
$|P_{ab}|=|Q_{ab}|$,
$|P_{bc}|=|Q_{bc}|$, and
$|P_c|=|Q_c|$.
Hence, 
$P-P_a+Q_a$,
$Q-Q_a+P_a$,
$P-P_{ab}+Q_{ab}$,
$Q-Q_{ab}+P_{ab}$,
$P-P_{bc}+Q_{bc}$,
$Q-Q_{bc}+P_{bc}$, 
$P-P_c+Q_c$ and~$Q-Q_c+P_c$ are longest paths.

Let~$H$ be an auxiliary graph given by
$V(H)=\{P'_a,P'_{ab},P'_{bc},
P'_c,Q'_a,Q'_{ab},Q'_{bc},Q'_c\}$
and
$E(H)=\{XY:\mbox{there is
a }X\mbox{-}Y\mbox{ path in }G' \}$.
By Lemma \ref{lemma:noDisjoint},
the next sets are independent sets in
$H$:~$\{P'_a,Q'_a,P'_{ab},Q'_{ab}\}$,
$\{P'_{ab},Q'_{ab},P'_{bc},Q'_{bc}\}$,
and
$\{P'_{bc},Q'_{bc}, P'_c,Q'_c\}$.
As~$G$ is 4-connected,
the graph~$G'=G \setminus S$
is connected, which implies that~$H$ is connected.
Thus, in~$H$, there is
a
$\{P'_a,Q'_a,P'_{bc},Q'_{bc}\}$-$\{P'_{ab},Q'_{ab},P'_{c},Q'_{c}\}$
path.
Hence, one of~$\{P'_aP'_c,
P'_aQ'_c,Q'_aP'_c,Q'_aQ'_c\}$
is an edge of~$H$. Without loss of
generality, we may assume that~$P'_aP'_c \in E(H)$.

This implies that there exists
a~$P'_a$-$P'_c$ path, say~$R$, in~$G'$. 
Let~$\{x\}= V(R) \cap V(P'_a)$
and~$\{y\}= V(R) \cap V(P'_c)$.
Let~$P_x$ and~$P_{xa}$ be the
corresponding subpaths of~$P_a$.
Let~$P_y$ and~$P_{yc}$ be the
corresponding subpaths of~$P_c$.
Then~$P-P_{xa}-P_y + R + Q_a$ and
$Q-Q_{a}+P_{xa} + R + P_y$
are both paths, whose lengths sum
$|P|+|Q|+2|R|$, a contradiction
(See Fig. \ref{fig:theorem-4-connected-4vertices}(a)).

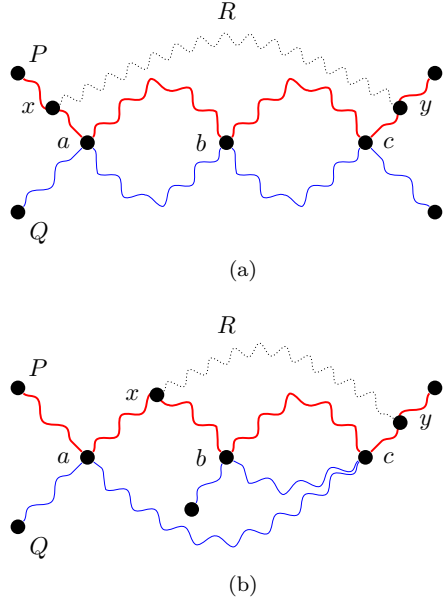
\begin{figure}[htp]
\centering
\subfloat[\label{fig:subim1}]{%
  \resizebox{.5\textwidth}{!}{\begin{tikzpicture}
[scale=0.5, label distance=3pt, every node/.style={fill,circle,inner sep=0pt,minimum size=6pt}]
 
    \node at
    (2,2)[myblue,label=left :$a$,fill=black, circle](a) {};
    \node at
    (6,2)[myblue,label=left :$b$,fill=black, circle](b) {};
     \node at
    (10,2)[myblue,label=right :$c$,fill=black, circle](c) {};
   
    \node at (0,4)[myblue,label=above right:$P$,fill=black, circle](a0) {};
    \node at (12,4)[myblue,label=above right:$ $,fill=black, circle](c0) {};
    \node at (12,0)[myblue,label=below right:$ $,fill=black, circle](c1) {};
    \node at (0,0)[myblue,label=below right :$Q$,fill=black, circle](a1) {};
  
    \draw [decorate, decoration=snake, segment length=0.6cm,color=red, thick] (a0) -- (a);
    \draw [decorate, decoration=snake, segment length=0.6cm,color=blue] (a1) -- (a);
    %\draw (c1) -- (d1);
    %\draw (d1) -- (s1);
    
    \draw [decorate, decoration=snake, segment length=0.5cm,color=red, thick] (a) .. controls (4, 4)  .. (b);
       \draw [decorate, decoration=snake, segment length=0.5cm,color=blue] (b) .. controls (4, 0)  .. (a);
      \draw[decorate, decoration=snake, segment length=0.5cm,color=red, thick] (b) .. controls (8, 4)  .. (c);      
      \draw[decorate, decoration=snake, segment length=0.5cm,color=blue] (c) .. controls (8, 0) .. (b);

    \draw[decorate, decoration=snake, segment length=0.6cm,color=red, thick] (c0) -- (c);

    \draw[decorate, decoration=snake, segment length=0.6cm,color=blue] (c1) -- (c);
    
  \node at
    (1,3)[myblue,label=left :$x$,fill=black, circle](x) {};
  \node at
    (11,3)[myblue,label=right :$y$,fill=black, circle](y) {};
    
      \node at
    (6,5)[myblue,label=above :$R$,fill=none, circle](R) {};
    
    % \draw (x) .. controls (2, 4.5) and (3, 4.5) .. (4,4.5) .. controls (8, 4.5) and (12, 4.5) .. (y);    
\draw [decorate, decoration=snake, densely dotted] (x) .. controls (6, 5.5) .. (y);

\end{tikzpicture}}
}\hfil
\subfloat[\label{fig:subim2}]{%
  \resizebox{.5\textwidth}{!}{\begin{tikzpicture}
[scale=0.5, label distance=3pt, every node/.style={fill,circle,inner sep=0pt,minimum size=6pt}]
 
    \node at
    (2,2)[myblue,label=left :$a$,fill=black, circle](a) {};
    \node at
    (6,2)[myblue,label=left :$b$,fill=black, circle](b) {};
     \node at
    (10,2)[myblue,label=right :$c$,fill=black, circle](c) {};
   
    \node at (0,4)[myblue,label=above right:$P$,fill=black, circle](a0) {};
    \node at (12,4)[myblue,label=above right:$ $,fill=black, circle](c0) {};
    \node at (5,0.5)[myblue,label=below right:$ $,fill=black, circle](c1) {};
    \node at (0,0)[myblue,label=below right :$Q$,fill=black, circle](a1) {};
      
        \draw [decorate, decoration=snake, segment length=0.6cm,color=red, thick] (a0) -- (a);
    \draw [decorate, decoration=snake, segment length=0.5cm,color=red, thick] (a) .. controls (4, 4)  .. (b);
      \draw[decorate, decoration=snake, segment length=0.5cm,color=red, thick] (b) .. controls (8, 4)  .. (c);      
    \draw[decorate, decoration=snake, segment length=0.6cm,color=red, thick] (c0) -- (c);      
      
      \draw[decorate, decoration=snake, segment length=0.5cm,color=blue] (c) .. controls (8, 1) .. (b);
             \draw [decorate, decoration=snake, segment length=0.5cm,color=blue] (c) .. controls (6, -1)  .. (a);
    \draw[decorate, decoration=snake, segment length=0.6cm,color=blue] (c1) -- (b);
        \draw [decorate, decoration=snake, segment length=0.6cm,color=blue] (a1) -- (a);
    
  \node at
    (4,3.8)[myblue,label=left :$x$,fill=black, circle](x) {};
  \node at
    (11,3)[myblue,label=right :$y$,fill=black, circle](y) {};
    
      \node at
    (6,5)[myblue,label=above :$R$,fill=none, circle](R) {};
    
    % \draw (x) .. controls (2, 4.5) and (3, 4.5) .. (4,4.5) .. controls (8, 4.5) and (12, 4.5) .. (y);    
\draw [decorate, decoration=snake, densely dotted] (x) .. controls (7, 5.5) .. (y);

\end{tikzpicture}}
}

\caption{Cases in the proof of Theorem \ref{theorem:4connected-4vertices}}
\label{fig:theorem-4-connected-4vertices}

\end{figure}

\item [Case 2.]
$acb$ is a subsequence in~$Q$.

It is easy to see that
$|P_a|=|Q_a|$ and
$|P_{bc}|=|Q_{bc}|$.
Hence, 
$P-P_a+Q_a$,
$Q-Q_a+P_a$,
$P-P_{ab}+Q_{ab}$, and
${Q-Q_{ab}+P_{ab}}$
are longest paths.
Let~$H$ be an auxiliary graph given by
$V(H)=\{P'_a,P'_{ab},P'_{bc},
P'_c,Q'_a,Q'_{ac},Q'_{cb},Q'_b\}$
and
$E(H)=\{XY:\mbox{there is
a }X\mbox{-}Y\mbox{ path in }G' \}$.
By Lemma \ref{lemma:noDisjoint},
the next sets are independent sets in
$H$:~$\{P'_a,Q'_a,P'_{ab},Q'_{ac}\}$,$\{P'_{ab},Q'_{b},P'_{bc},Q'_{bc}\}$, and $\{P'_{bc},Q'_{bc}, P'_c,Q'_{ac}\}$.
Thus, as~$H$ is connected,~$P'_{ab}P'_c \in E(H)$.
This implies that there exists
a~$P'_{ab}$-$P'_c$ path, say~$R$, in~$G'$. 
Let~${\{x\}= V(R) \cap V(P'_{ab})}$
and~${\{y\}= V(R) \cap V(P'_c)}$.
Let~$P_{ax}$ and~$P_{xb}$ be the corresponding
subpaths of~$P_{ab}$.
Let~$P_y$ and~$P_{yc}$ be the two
corresponding subpaths of~$P_c$.
Then~$P-P_{xb}-P_{bc}-P_y + R + Q_b+Q_{bc}$ and
$Q-Q_b-Q_{bc}+P_{xb} + R + P_{bc}+P_y$
are both paths, whose lengths sum
$|P|+|Q|+2|R|$, a contradiction
(See Fig. \ref{fig:theorem-4-connected-4vertices}(b)).

\end{description}
With that, we finish the proof of Theorem \ref{theorem:4connected-4vertices}.
\end{proof}

%\begin{figure}
%\begin{subfigure}{.5\textwidth}
%  \centering
%  % include first image
%  \resizebox{.5\textwidth}{!}{\input{figure-main-theorem-1}}
%  \label{fig:sub-first}
%\end{subfigure}
%\begin{subfigure}{.5\textwidth}
%  \centering
%  % include second image
% \resizebox{.5\textwidth}{!}{\input{figure-main-theorem-2}}
%  \label{fig:sub-second}
%\end{subfigure}
%%\includegraphics[width=\textwidth]{fig1.eps}
%\caption{The graph used in the construction of Theorem %\ref{theorem:main-theorem}.} \label{figure:main-theorem-1}
%\end{figure}

\section{Tight families}
As mentioned by Hippchen
\cite[Figure 2.5]{Hippchen08}, the graph
$K_{k,2k+2}$ makes the conjecture tight.
%Hippchen mentioned \cite[Figure 2.5]{Hippchen08}, that in the graph $K_{k,2k+2}$, there exists a pair of longest paths intersecting each other in
%exactly $k$ vertices.
%But it is a fact that $K_{k,k+2}$ is not $k$-connected, so this set of examples does not make the conjecture tight.
In this section, we show that in fact there is
an infinite family of graphs, for every $k$, that
make Hippchen's Conjecture tight.

\begin{theorem}\label{theorem:tightFamilies}
For every $k$-connected graph, there
exists an infinite family of graphs
with a pair of longest paths intersecting each other
in exactly $k$ vertices.
\end{theorem}
\begin{proof}
Let $\ell$ be an arbitrary positive integer.
Let $S=\{s_1,s_2,\ldots, s_k\}$,
and, for every $i\in [k+1]$, let $X_i=\{a_{i1},a_{i2},\ldots,a_{i\ell}\}$ and
$Y_i=\{b_{i1},b_{i2},\ldots, b_{i\ell}\}$.
Let $V(G)=S \cup \{X_i: i \in [k+1]\} \cup \{Y_i: i \in [k+1]\}$,
and 
$E(G)=
\{sv: s \in S, v \in V(G) \setminus S \}
\cup \{a_{ij}a_{i(j+1)}: i \in [k+1], j \in [\ell-1] \}
\cup \{b_{ij}b_{i(j+1)}: i \in [k+1], j \in [\ell-1]  \}
$ 
%$E(G)=
%\{s_ia_j: i \in [k], j \in [k+1] \} %\cup
%\{s_ib_j: i \in [k], j \in [k+1] \}
%\cup
%\{s_ic_j: i \in [k], j \in [k+1] \}
%\cup
%\{s_id_j: i \in [k], j \in [k+1] \}
%\cup
%\{a_ib_i: i \in [k+1]\}
%\cup
%\{c_id_i: i \in [k+1] \}
%$
(See Fig. \ref{figure:tight-families}).

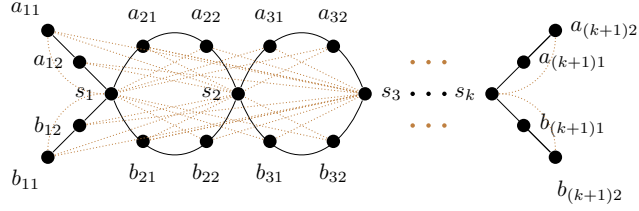
\begin{figure}
\centering
\resizebox{.7\textwidth}{!}{%defs-tree-dec-4
\begin{tikzpicture}
[scale=0.5, label distance=3pt, every node/.style={fill,circle,inner sep=0pt,minimum size=6pt}]
 
    \node at
    (2,2)[myblue,label=left :$s_1$,fill=black, circle](s1) {};
    \node at
    (6,2)[myblue,label=left :$s_2$,fill=black, circle](s2) {};
     \node at
    (10,2)[myblue,label=right :$s_3$,fill=black, circle](s3) {};
     \node at
    (11.5,2)[myblue,minimum size=2pt,label=above :$ $,fill=black, circle](dot1) {};
      \node at
    (12,2)[myblue,minimum size=2pt,label=above :$ $,fill=black, circle](dot2) {};
      \node at
    (12.5,2)[myblue,minimum size=2pt,label=above :$ $,fill=black, circle](dot3) {};
    
      \node at (11.5,3)[brown,minimum size=2pt,label=above :$ $,fill=brown, circle](dotbrown1) {};
      \node at
    (12,3)[brown,minimum size=2pt,label=above :$ $,fill=brown, circle](dotbrown2) {};
      \node at
    (12.5,3)[brown,minimum size=2pt,label=above :$ $,fill=brown, circle](dotbrown3) {};
    
     \node at (11.5,1)[brown,minimum size=2pt,label=above :$ $,fill=brown, circle](dotbrown4) {};
      \node at
    (12,1)[brown,minimum size=2pt,label=above :$ $,fill=brown, circle](dotbrown5) {};
      \node at
    (12.5,1)[brown,minimum size=2pt,label=above :$ $,fill=brown, circle](dotbrown5) {};
    
    \node at
    (14,2)[myblue,label=left :$s_k$,fill=black, circle](sk) {};
    \node at (0,4)[myblue,label=above left:$a_{11}$,fill=black, circle](a1) {};
    \node at (16,4)[myblue,label=right:$a_{(k+1)2}$,fill=black, circle](bkm1) {};
    \node at (16,0)[myblue,label=below right:$b_{(k+1)2}$,fill=black, circle](dkm1) {};
    \node at (0,0)[myblue,label=below left :$b_{11}$,fill=black, circle](c1) {};
    \node at
    (1,3)[myblue,label=left :$a_{12}$,fill=black, circle](b1) {};
    \node at
    (15,3)[myblue,label=right :$a_{(k+1)1}$,fill=black, circle](akm1) {};
    \node at
    (1,1)[myblue,label=left :$b_{12}$,fill=black, circle](d1) {};\node at
    (15,1)[myblue,label=right :$b_{(k+1)1}$,fill=black, circle](ckm1) {};
    
     \node at
    (3,3.5)[myblue,label=above :$a_{21}$,fill=black, circle](a2) {};
    \node at
    (5,3.5)[myblue,label=above :$a_{22}$,fill=black, circle](b2) {};
         \node at
    (3,0.5)[myblue,label=below :$b_{21}$,fill=black, circle](c2) {};
    \node at
    (5,0.5)[myblue,label=below :$b_{22}$,fill=black, circle](d2) {};
    
         \node at
    (7,3.5)[myblue,label=above :$a_{31}$,fill=black, circle](a3) {};
    \node at
    (9,3.5)[myblue,label=above :$a_{32}$,fill=black, circle](b3) {};
         \node at
    (7,0.5)[myblue,label=below :$b_{31}$,fill=black, circle](c3) {};
    \node at
    (9,0.5)[myblue,label=below :$b_{32}$,fill=black, circle](d3) {};

    \draw (a1) -- (b1);
    \draw (b1) -- (s1);
    \draw (c1) -- (d1);
    \draw (d1) -- (s1);
    
    \draw (s1) .. controls (3, 4.5) and (5, 4.5) .. (s2);
       \draw (s1) .. controls (3, -0.5) and (5, -0.5) .. (s2);
      \draw (s2) .. controls (7, 4.5) and (9, 4.5) .. (s3);      
      \draw (s2) .. controls (7, -0.5) and (9, -0.5) .. (s3);
    
     \draw (akm1) -- (bkm1);
    \draw (bkm1) -- (sk);
    \draw (ckm1) -- (dkm1);
    \draw (dkm1) -- (sk);
    
   \draw [densely dotted] [color=brown] [style = thin][line width=0.45pt]  (a1) to [bend left = -45] (s1);
 \draw [densely dotted] [color=brown] [style = thin][line width=0.45pt]  (a1) to (s2);  
  \draw [densely dotted] [color=brown] [style = thin][line width=0.45pt]  (a1) to (s3);  
   \draw [densely dotted] [color=brown] [style = thin][line width=0.45pt]  (a2) to (s2);  
  \draw [densely dotted] [color=brown] [style = thin][line width=0.45pt]  (a2) to (s3);   
  \draw [densely dotted] [color=brown] [style = thin][line width=0.45pt]  (a2) to (s3);  
     \draw [densely dotted] [color=brown] [style = thin][line width=0.45pt]  (b1) to (s2);
 \draw [densely dotted] [color=brown] [style = thin][line width=0.45pt]  (b1) to (s3);   
   \draw [densely dotted] [color=brown] [style = thin][line width=0.45pt]  (d1) to (s2);
  \draw [densely dotted] [color=brown] [style = thin][line width=0.45pt]  (d1) to (s3);
    \draw [densely dotted] [color=brown] [style = thin][line width=0.45pt]  (b2) to (s1);   
  \draw [densely dotted] [color=brown] [style = thin][line width=0.45pt]  (b2) to (s3);
    \draw [densely dotted] [color=brown] [style = thin][line width=0.45pt]  (a3) to (s1);   
  \draw [densely dotted] [color=brown] [style = thin][line width=0.45pt]  (a3) to (s3);    
      \draw [densely dotted] [color=brown] [style = thin][line width=0.45pt]  (b3) to (s1);   
  \draw [densely dotted] [color=brown] [style = thin][line width=0.45pt]  (b3) to (s2);

   \draw [densely dotted] [color=brown] [style = thin][line width=0.45pt]  (c1) to [bend left = 45] (s1);
 \draw [densely dotted] [color=brown] [style = thin][line width=0.45pt]  (c1) to (s2);  
  \draw [densely dotted] [color=brown] [style = thin][line width=0.45pt]  (c1) to (s3);  
   \draw [densely dotted] [color=brown] [style = thin][line width=0.45pt]  (c2) to (s2);  
  \draw [densely dotted] [color=brown] [style = thin][line width=0.45pt]  (c2) to (s3);   
  \draw [densely dotted] [color=brown] [style = thin][line width=0.45pt]  (c2) to (s3);  
    \draw [densely dotted] [color=brown] [style = thin][line width=0.45pt]  (d2) to (s1);   
  \draw [densely dotted] [color=brown] [style = thin][line width=0.45pt]  (d2) to (s3);
    \draw [densely dotted] [color=brown] [style = thin][line width=0.45pt]  (c3) to (s1);   
  \draw [densely dotted] [color=brown] [style = thin][line width=0.45pt]  (c3) to (s3);    
      \draw [densely dotted] [color=brown] [style = thin][line width=0.45pt]  (d3) to (s1);   
  \draw [densely dotted] [color=brown] [style = thin][line width=0.45pt]  (d3) to (s2);
  
    \draw [densely dotted] [color=brown] [style = thin][line width=0.45pt]  (sk) to [bend left = -45] (bkm1);
      \draw [densely dotted] [color=brown] [style = thin][line width=0.45pt]  (sk) to [bend left = 45] (dkm1);

\end{tikzpicture}}

\caption{The graph used in the construction of Theorem \ref{theorem:tightFamilies} in the case $\ell=2$.} \label{figure:tight-families}
\end{figure}

Note that $S$ is a separator in $G$, and that every component of
$G-S$ has size $\ell$. Hence, any path in $G$ has at most $k+\ell(k+1)$ vertices.
Then $P=a_{11}a_{12} \cdots a_{1\ell}s_1a_{21}a_{22} \cdots a_{2\ell}s_2 \cdots a_{k1}a_{k2} \cdots a_{k\ell}s_ka_{(k+1)1}a_{(k+1)2} \cdots a_{(k+1)\ell}$
and \newline
$Q=b_{11}b_{12} \cdots b_{1\ell}s_1b_{21}b_{22} \cdots b_{2\ell}s_2 \cdots b_{k1}b_{k2} \cdots b_{k\ell}s_kb_{(k+1)1}b_{(k+1)2} \cdots b_{(k+1)\ell}$
are both longest paths,
intersecting each other in exactly $k$ vertices.

To finish the proof, we will show that $G$ is $k$-connected.
Suppose by contradiction that $G$ has a separator $S'$ of cardinality at most $k-1$.
As $|S|=k$, there exists a vertex $s \in S \setminus S'$.
But then $G-S'$ is connected, as $s$ is a universal vertex in $G-S'$,
a contradiction.
%To finish the proof, it is easy to see that $G$ is $k$-connected, as, %for all $u,v \notin S$
%$\{us_iv: i \in [k]\}$
%is a family of 
%$k$ internally disjoint paths
%from $u$ to $v$.
\end{proof}

%It is easy to see that the graph constructed in Theorem \ref{theorem:tightFamilies} can be adapted to construct
%an infinite class of families for every $k$. For this, it
%suffices to add new vertices to the sets $X_i$ and $Y_i$,
%and join them to every vertex in $S$.

\section{Conclusions and future work}
 \label{section:conclusion}

In this paper, we showed that every pair of longest paths in a~$k$-connected graph intersect each other in at least~$(8k-n+2)/5$ vertices.
A direct corollary of this result is that,
if~$k \geq n/3$, then every pair of longest paths intersect in at least~$k$ vertices; and, if~$k \geq n/4$, then every pair of longest paths intersect in at least~$\frac{4(k+1)}{5}$ vertices. 
For the latter, we propose the following weak version of Hippchen's Conjecture.

\begin{conjecture}\label{conj:densekn4}
Let~$G$ be a~$k$-connected graph on~$n$ vertices.
If~$k \geq n/4$ then every pair of longest paths intersect in at least~$k$ vertices.
\end{conjecture}

In this paper, we also showed that,
in a 4-connected graph,
every pair of longest paths intersect in at least 4 vertices. We believe the result for 4-connected graphs can be extended to 5-connected graphs.

\begin{conjecture}\label{conj:sparse5}
In a 5-connected graph, every pair of longest paths intersect each other in at least five vertices.
\end{conjecture}

We believe Conjectures \ref{conj:densekn4} and \ref{conj:sparse5} can be approached with new or similar techniques as the ones presented here.
We also think the techniques presented here can be adapted to show similar results for cycles instead of paths; and, conversely, that the techniques used  by Chen \etal~\cite{Chen98} can be adapted to show similar or strong results for paths.

\bibliography{bibliografia}

\end{document}